\documentclass[12pt]{amsart}
\usepackage{color}
\usepackage[utf8]{inputenc}
\usepackage{amsfonts,amsmath}
\usepackage{amssymb,latexsym}
\usepackage{etoolbox}
\patchcmd{\section}{\scshape}{\bfseries\Large}{}{}
\makeatletter
\renewcommand{\@secnumfont}{\bfseries\Large}
\makeatother
\newtheorem{theorem}{Theorem}
\newtheorem{lemma}[theorem]{Lemma}

\newtheorem{corollary}[theorem]{Corollary}

\theoremstyle{remark}

\newcommand{\R}{\mathbb{R}}
\renewcommand{\leq}{\leqslant}
\renewcommand{\geq}{\geqslant}
\def\N{\mathbb{N}}
\def\al{\alpha}
\def\PP{\mathbb{P}}
\DeclareMathOperator{\dive}{div}
\DeclareMathOperator{\curl}{curl}

\def\ep{\varepsilon}
\def\om{\omega}
\def\Om{\Omega}
\def\oma{\omega_\alpha}
\newcommand\nl[2]{\|#2\|_{L^{#1}}}
\newcommand\nh[2]{\|#2\|_{H^{#1}}}
\newcommand\lip[1]{\|\nabla #1\|_{L^{\infty}}}
\begin{document}
\title{From second grade fluids to the Navier-Stokes equations}
\author{A. V. Busuioc}

\address[A.V. Busuioc]{Université de Lyon, Université de Saint-Etienne  --
CNRS UMR 5208 Institut Camille Jordan --
Faculté des Sciences --
23 rue Docteur Paul Michelon --
42023 Saint-Etienne Cedex 2, France}
\email{valentina.busuioc@univ-st-etienne.fr}

\begin{abstract}
We consider the limit $\al\to0$ for a second grade fluid on a bounded domain with Dirichlet boundary conditions. We show convergence towards a solution of the Navier-Stokes equations under two different types of hypothesis on the initial velocity $u_0$. If the product $\nl2{u_0}\nh1{u_0}$ is sufficiently small we prove global-in-time convergence. If there is no smallness assumption  we obtain local-in-time convergence up to the time $C/\nh1{u_0}^4$.
\end{abstract}

\maketitle

\section{Introduction}
We consider in this paper the incompressible second grade fluid equations:
\begin{equation}\label{maineq}
  \partial_t (u-\al\Delta u)-\nu\Delta u +u\cdot\nabla(u-\al\Delta u)+\sum_j(u-\al\Delta u)_j\nabla u_j=-\nabla p, \qquad \dive u=0,
\end{equation}
where $\alpha$ and $\nu$ are some non-negative constants. The fluid is assumed to be enclosed in a bounded smooth region $\Omega$ of $\R^3$ and the homogeneous Dirichlet boundary conditions are imposed
\begin{equation}\label{2}
  u(t,\cdot)\bigl|_{\partial\Omega}=0\qquad \text{for all }t\geq0.
\end{equation}
The initial value problem is considered and we denote by $u_0$ the initial velocity:
\begin{equation}\label{3}
  u(0,x)=u_0(x).
\end{equation}

The equations \eqref{maineq} were deduced in \cite{dunn_thermodynamics_1974} from physical principles. Let us just mention here that the second grade fluids are characterized by the following fact: the stress tensor is a polynomial of degree two in the first two Rivlin-Ericksen tensors which are the deformation tensor $D$ and the tensor $(\partial_t+u\cdot\nabla)D$. The vanishing viscosity case $\nu=0$ is also known under the name $\al$--Euler or Euler--$\al$ equations and was later obtained via an averaging procedure performed on the classical incompressible Euler equations.

Two main boundary conditions were used for \eqref{maineq} in the mathematical literature: the no-slip boundary conditions and the frictionless slip Navier boundary conditions where the fluid is allowed to slip on the boundary without friction. The second boundary condition is more complex but allows for better mathematical results; also it has less physical relevance. The classical well-posedness results for \eqref{maineq} are the following (see \cite{cioranescu_weak_1997,cioranescu_existence_1984,galdi_existence_1993,galdi_further_1994} for the Dirichlet boundary conditions and \cite{busuioc_second_2003} for the Navier boundary conditions):
\begin{itemize}
\item In dimension two there exists a unique global $H^3$ solution if $u_0\in H^3$.
\item In dimension three there exists a unique local $H^3$ solution if $u_0\in H^3$. The solution is global if $u_0$ is small in $H^3$.
\end{itemize}
We call $H^3$ solution a divergence free vector field verifying the boundary conditions and the PDE \eqref{maineq} and who is bounded in time (up to time $t=0$) with values in $H^3(\Om)$. Let us also mention the paper \cite{bresch_existence_1998} where solutions in $W^{2,p}$, $p>3$, are constructed.

Let us observe that when $\al=0$ relation \eqref{maineq} becomes the Navier-Stokes equations
\begin{equation}\label{NS}
\partial_t u-\nu\Delta u+u\cdot\nabla u=-\nabla p  
\end{equation}
and when $\al=\nu=0$ it becomes the Euler equations. It is interesting to know if the solutions of \eqref{maineq} converge to the solutions of the limit equation when $\al\to0$ and $\nu>0$ is fixed or when $\al,\nu\to0$. This was already studied in several papers as we shall see below.

Let us first mention that in the absence of boundaries one can obtain $H^3$ estimates uniform in $\al$ and $\nu$ in both dimensions two and three and pass to the limit. This was performed in \cite{linshiz_convergence_2010}, see also \cite{busuioc_incompressible_2012} for a simpler proof. But such a result cannot hold true on domains with boundaries. Indeed, if the solutions of \eqref{maineq}-\eqref{2} are bounded in $H^3$ uniformly in $\al$ then one can easily pass to the limit $\al\to0$ and obtain at the limit a solution of the Navier-Stokes equations which must also be bounded in $H^3$. For such a solution to exist, the initial data must verify a compatibility condition. Indeed, one can apply the Leray projector to \eqref{NS} to obtain that $\partial_t u-\nu\PP\Delta u+\PP(u\cdot\nabla u)=0$. Since $u$ vanishes at the boundary, so does $\partial_t u$. We infer that $-\nu\PP\Delta u+\PP(u\cdot\nabla u)=0$ at the boundary. Observe that $u$ being in $H^3$ implies that these two terms are in $H^1$ so the trace at the boundary makes sense. By time continuity we infer that the initial data must verify the compatibility condition  $\nu\PP\Delta u_0=\PP(u_0\cdot\nabla u_0)$ at the boundary. This is of course in general not verified if we only assume  that $u_0\in H^3$ is divergence free and vanishing on the boundary.

We review now the results available for domains with boundary.

Concerning the limit $\al,\nu\to0$, we proved in \cite{busuioc_incompressible_2012} the expected convergence in 2D for weak $H^1$ solutions in the case of the Navier boundary conditions. We also proved the convergence in 3D but under the additional hypothesis that the solutions exist on a time interval independent of $\al$ and $\nu$.  We proved in \cite{busuioc_uniform_2016} that the hypothesis of existence of a uniform time existence is verified if $\nu=0$ and $\al\to0$. In the case of the Dirichlet boundary conditions, there is only the paper \cite{lopes_filho_convergence_2015} which shows convergence in 2D. The case of Dirichlet boundary conditions in 3D is open.

Concerning the limit $\al\to0$ and $\nu>0$ fixed, it was proved in \cite{iftimie_remarques_2002-1} the expected convergence for weak $H^1$ solutions. That result is stated for $\Om=\R^n$, $n=2,3$, but the proof relies only on energy estimates and standard compactness arguments so it goes through to bounded domains without difficulty (Dirichlet and  Navier boundary conditions likewise), see also \cite[Remarque 4]{iftimie_remarques_2002-1}. There is however a major drawback to the result of \cite{iftimie_remarques_2002-1}: the author assumes that the sequence of weak $H^1$ solutions exist on a time interval independent of $\al$ and proves convergence on any such time interval. This raises the question of proving the existence of such a uniform time interval. Let us also mention the paper \cite{arada_convergence_2016} where the author considers the 2D case with Navier boundary conditions and shows stronger convergence of solutions together with some estimates for the rate of convergence. Observe however that in dimension two the solutions are global in time so the hypothesis assumed in \cite{iftimie_remarques_2002-1} that the solutions exist on a uniform time interval is automatically satisfied.

In this paper we aim to prove that the hypothesis of \cite{iftimie_remarques_2002-1} about the existence of solutions on a uniform time interval is verified in various situations for Dirichlet boundary conditions. Since global existence of solutions holds true in dimension two for Dirichlet boundary conditions, in 2D the hypothesis of \cite{iftimie_remarques_2002-1} is automatically satisfied and the problem is settled. So we restrict ourselves to the 3D case. Recall that the Navier-Stokes equations in dimension three are locally well-posed for large data and  globally well-posed for small data. Likewise, we will prove two results on the uniform time of existence: a local result  for large data and a global result for small data. 

Surprisingly, we find that if the $H^1$ norm of the initial data is sufficiently small and if $\al$ is sufficiently small too, then the solutions of \eqref{maineq} are global. Let us emphasize that the smallness of the initial data is measured only in the $H^1$ norm (in fact in a weaker space, see Theorem \ref{globalthm}  in the next section) and not in the $H^3$ norm as required by the classical global well-posedness result for \eqref{maineq}. It was expected to find that, if the $H^1$ norm of the initial data is small then the maximal time of existence of \eqref{maineq} goes to infinity as $\al\to0$. Indeed the limit equations, \textit{i.e.} the Navier-Stokes equations, are globally well-posed for small $H^1$ initial data and so we expect convergence for the times of existence of solutions. But we did not expect the maximal time of existence to actually be infinite if $\al$ is sufficiently small.

For large data, we prove in particular that if $\al$ is sufficiently small then the time existence of the solution has a lower bound that depends only on the $H^2$ norm of the initial velocity. This should be compared with previous results on local existence of solutions where the time of existence depends on the $H^3$ norm or the $W^{2,p}$, $p>3$, norm of the initial velocity. When $\al\to0$ we obtain convergence of solutions of \eqref{maineq} towards solutions of \eqref{NS} up to the time $C/\nh1{u_0}^4$. 

The plan of the paper is the following. In the next section we state and comment our results. In Section \ref{globalsect} we prove the global result for small data. In the last section, we prove our ``uniform local existence'' result.

\section{Statement of the results}

Since we are interested in the limit $\al\to0$ with $\nu$ fixed, we will assume throughout this paper that $\al\leq1$ and $\nu>0$.

Let us first state the result of \cite{iftimie_remarques_2002-1} which motivates the present work.
\begin{theorem}[see \cite{iftimie_remarques_2002-1}] \label{teorema}
Let $\nu>0$, $T>0$ and $u_\al$ some $H^1$ solutions of \eqref{maineq} defined up to the time $T$ such that $\nh1{u_\al(0)}\leq C\al^{-\frac12}$ and $u_\al(0)$ converges weakly in $L^2$ to some $u_0$. Then there exists a weak Leray solution $u$ of the Navier-Stokes equations with initial velocity $u_0$ such that, after extraction of a sub-sequence,
\begin{equation*}
u_\al\rightharpoonup u
\text{ in } L^\infty_{loc}([0,T);L^2) \text{ weak* and in } 
L^2_{loc}([0,T);H^1) \text{ weak.}
\end{equation*}
\end{theorem}

Our result about global existence of solutions reads as follows.
\begin{theorem}\label{globalthm}
There exists some small constant $\ep=\ep(\Om)$ depending only on $\Om$ such that if the initial velocity $u_0$ belongs to $H^3(\Om)$, is divergence free, vanishes at the boundary and verifies the following smallness conditions
\begin{gather}
\nl2{u_0}\nh1{u_0}\leq \ep^2\nu^2\label{small1}\\
\intertext{and}
\nl2{u_0}\nh2{u_0}\leq \ep^2\nu^2\al^{-\frac12},\qquad
\nh1{u_0}\nh2{u_0}\leq\ep^2\nu^2\al^{-1}\quad  \text{and}\quad \nh3{u_0}  \leq\ep\nu\al^{-\frac54},\label{small2}
\end{gather}
then the solution $u$ of \eqref{maineq}-\eqref{3} is global: $u\in L^\infty(\R_+;H^3(\Om))$.
\end{theorem}

The important thing to observe here is that the smallness conditions in \eqref{small2} involve only negative powers of $\al$, so they disappear when $\al\to0$. Only the smallness condition \eqref{small1} subsists, and this is in accordance with the classical global well-posedness results for small data for the Navier-Stokes equations. Note that condition \eqref{small1} is scaling invariant and that it implies the smallness of the $H^{\frac12}$ norm of $u_0$ by the interpolation inequality $\nh{\frac12}{u_0}\leq\nl2{u_0}^{\frac12}\nh1{u_0}^{\frac12}$. This in turn gives the existence of a unique global solution of \eqref{2}-\eqref{NS} by the celebrated result of Fujita and Kato  \cite{fujita_navier-stokes_1964}. We have the following immediate corollary.
\begin{corollary}
Suppose that $u_0\in H^3(\Om)$ is divergence free and vanishes on the boundary. There exists some small constant $\ep=\ep(\Om)$ depending only on $\Om$ and some constant $\al_0>0$ such that if
 \begin{equation*}
\nl2{u_0}\nh1{u_0}\leq\ep^2\nu^2\qquad\text{and}\qquad\al\leq\al_0   
 \end{equation*}
then the $H^3$ solution of \eqref{maineq}-\eqref{3} is globally defined.
\end{corollary}

Combining this corollary with Theorem \ref{teorema} we immediately obtain the global convergence of solutions of \eqref{maineq} towards solutions of the Navier-Stokes equations if we assume that the initial velocity $u_0$ does not depend on $\al$, belongs to $H^3$ and is small in $H^1$. This is just an example of convergence result, more general results can be obtained by allowing the initial data to depend on $\al$ and combining Theorems \ref{teorema} and \ref{globalthm}.  

Finally, let us conclude our results on global solutions with the observation that the smallness conditions \eqref{small1} and \eqref{small2}
are implied by the simpler but less general condition:
\begin{equation*}
 \nh1{u_0}\leq\ep\nu,\qquad   \nh2{u_0}\leq\ep\nu\al^{-\frac12}\qquad  \text{and}\qquad \nh3{u_0}  \leq\ep\nu\al^{-\frac54}.
\end{equation*}
Indeed, if the above relation holds true one can readily obtain \eqref{small1} and \eqref{small2} simply by estimating  $\nl2{u_0}\leq\nh1{u_0}$.

\medskip

We state now our result on the ``uniform local existence of solutions''.
\begin{theorem}\label{local}
There exist two constants $\ep=\ep(\Om)$ and $K=K(\Om)$ depending only on $\Omega$ such that if $u_0\in H^3(\Om)$ is divergence free and vanishes on the boundary and 
\begin{equation}\label{hyplocal}
  \nh1{u_0}\leq \ep\nu\al^{-\frac14}\qquad\text{and}\qquad \nh3{u_0}\leq\ep\nu\al^{-\frac54}
\end{equation}
then there exists a $H^3$ solution $u\in L^\infty(0,T;H^3(\Om))$ of \eqref{maineq}-\eqref{3} which is defined at least up to the time
\begin{equation*}
  T=\frac{\nu^3}{K(\nh1{u_0}+\sqrt\al\nh2{u_0})^4}.
\end{equation*}
\end{theorem}

Let us emphasize again that in the above theorem the time of existence of the solution depends on the $H^2$ norm of the initial data, but not on the $H^3$ norm or the $W^{2,p}$, $p>3$, norm as in the previous results on local existence of solutions. 

Another important thing to note is that both the smallness conditions for $\nh1{u_0}$ and $\nh3{u_0}$ involve negative powers of $\al$. So they disappear when taking the limit $\al\to0$. More precisely, we have the following corollary.
\begin{corollary}
 Suppose that $u_0\in H^3(\Om)$ is divergence free and vanishes on the boundary. Let $T_\al$ be the maximal time of existence of the $H^3$ solution $u$ of \eqref{maineq}-\eqref{3}, \textit{i.e.}
$$u\in L^\infty_{loc}([0,T_\al);H^3(\Om))\setminus L^\infty(0,T_\al;H^3(\Om)).$$ 
Then
 \begin{equation*}
\liminf_{\al\to0}T_\al\geq \frac{\nu^3}{K\nh1{u_0}^4}.
 \end{equation*}
\end{corollary}
 
Combining this corollary with Theorem \ref{teorema} yields, under the assumptions of the corollary, the convergence of the solutions of \eqref{maineq} towards the solutions of the Navier-Stokes equations up to the time $T= \frac{\nu^3}{K\nh1{u_0}^4}$.

\section{Global existence of solutions}
\label{globalsect}

In this section we  show Theorem \ref{globalthm}. In the sequel, we will denote by $C$ a generic constant which depends only on $\Om$ and whose value can change from one line to another. We will use the standard notation for the $H^m$ norms
\begin{equation*}
  \|u\|_{H^m}=\Bigl(\sum_{\beta\in\N^3,|\beta|\leq m}\int_\Om|\partial^\beta u(x)|^2\,dx\Bigr)^{\frac12}.
\end{equation*}

We denote by $\PP$ the Leray projector, \textit{i.e.} the $L^2$ orthogonal projector on the subspace of divergence free and tangent to the boundary vector fields.

\medskip

We will prove that there exists  a sufficiently small constant $\ep_1=\ep_1(\Om)$ depending only on $\Om$ such that if 
\begin{equation}\label{cond}
\lip u\leq\frac\nu{2\al}\qquad\text{and}\qquad \nl3u\leq\ep_1 \nu
\end{equation}
on some time interval $[0,T]$, then the two quantities above are even smaller:
\begin{equation}\label{condimpl}
\lip u\leq\frac\nu{4\al}\qquad\text{and}\qquad \nl3u\leq\frac{\ep_1 \nu}2  
\end{equation}
on the same time interval $[0,T]$. By time continuity, this implies that \eqref{cond} never breaks down if it is verified at the initial time and all the estimates that follow hold true globally in time. Let us now observe that \eqref{cond} holds true at the initial time if $\ep$ and $\ep_1$ are sufficiently small independently of $\al$ and $\nu$. We use the Gagliardo-Nirenberg inequality and \eqref{small1} to write
\begin{equation*}
\nl3{u_0}\leq C\nl2{u_0}^{\frac12}\nh1{u_0}^{\frac12} \leq C\ep\nu< \ep_1\nu 
\end{equation*}
provided that $C\ep<\ep_1$. This is a condition that we assume in what follows. Next, we use an interpolation inequality and relation \eqref{small2} to estimate
\begin{equation*}
\nh1{u_0}\leq \nl2{u_0}^{\frac12}\nh2{u_0}^{\frac12}\leq \ep\nu\al^{-\frac14}.  
\end{equation*}
From the Gagliardo-Nirenberg inequality and \eqref{small2}  we deduce that
\begin{equation*}
  \lip{u_0}\leq C\nh1{u_0}^{\frac14}\nh3{u_0}^{\frac34}\leq C\frac{\ep\nu}\al< \frac\nu{2\al}
\end{equation*}
provided that $\ep<\frac1{2C}$ which is another condition that we will assume in the sequel. We conclude from the above estimates that \eqref{cond} holds true at the initial time with strict inequality provided that $C\ep<\ep_1$ and $\ep<\frac1{2C}$.

\medskip

We assume in the following that \eqref{cond} holds true. Because we need a precise dependence on $\al$ of the constants, we have to split the estimates for $u$ in three different parts: $H^1$ estimates, $H^2$ estimates and $H^3$ estimates. The estimates below should be viewed as \textit{a priori} estimates. They can be turned into rigorous estimates via a standard Galerkin approximation procedure.

\subsection*{$H^1$ estimates.} We multiply \eqref{maineq} by $u$ and integrate in space to obtain
\begin{equation*}
  \frac12 \partial_t(\nl2u^2+\al\nl2{\nabla u}^2)+\nu\nl2{\nabla u}^2=0
\end{equation*}
so, after an integration in time,
\begin{equation}\label{h1est}
\nl2{u(t)}^2+\al\nl2{\nabla u(t)}^2+2\nu\int_0^t\nl2{\nabla u(s)}^2\,ds= \nl2{u_0}^2+\al\nl2{\nabla u_0}^2. 
\end{equation}

\subsection*{$H^2$ estimates.} To perform $H^2$ estimates on \eqref{maineq} we need to write it under a different form. First we remove the pressure by applying the Leray projector $\PP$. We obtain the following equivalent equation:
\begin{equation*}
\partial_t (u-\al\PP\Delta u)-\nu\PP\Delta u +\PP\bigl[u\cdot\nabla(u-\al\Delta u)\bigr]+\PP\bigl[\sum_j(u-\al\Delta u)_j\nabla u_j\bigr]=0.  
\end{equation*}

We observe first that
\begin{equation*}
 \PP \bigl(\sum_j u_j\nabla u_j\bigr)=\frac12\PP\nabla(|u|^2)=0.
\end{equation*}

Next, we have that 
\begin{equation*}
  u\cdot\nabla\nabla q+\sum_j\partial_j q\nabla u_j=\nabla(u\cdot\nabla q)
\end{equation*}
is a gradient so
\begin{equation*}
  \PP\Bigl( u\cdot\nabla\nabla q+\sum_j\partial_j q\nabla u_j\Bigr)=0
\end{equation*}
for any $q$. Recalling that $\PP\Delta u-\Delta u$ is a gradient, we infer that  \eqref{maineq} can be written under the following equivalent form:
\begin{equation*}
\partial_t (u-\al\PP\Delta u)-\nu\PP\Delta u +\PP\bigl[u\cdot\nabla(u-\al\PP\Delta u)\bigr]-\al\PP\bigl[\sum_j(\PP\Delta u)_j\nabla u_j\bigr]=0.  
\end{equation*}

We multiply the above equation by $-\PP\Delta u$ and integrate in space. We recall that $\PP$ is a self-adjoint projector. We have that $\PP^2=\PP$ and $\PP u=u$ because $u$ is divergence free and tangent to the boundary. We infer that
\begin{align*}
\frac12\partial_t(\nl2{\nabla u}^2+&\al\nl2{\PP\Delta u}^2)+\nu\nl2{\PP\Delta u}^2\\
&=\int_\Om \PP\bigl[u\cdot\nabla(u-\al\PP\Delta u)\bigr]\cdot \PP\Delta u 
-\al\int_\Om \PP\bigl[\sum_j(\PP\Delta u)_j\nabla u_j\bigr] \cdot \PP\Delta u \\
&=\int_\Om u\cdot\nabla(u-\al\PP\Delta u)\cdot \PP\Delta u 
-\al\int_\Om \sum_j(\PP\Delta u)_j\nabla u_j \cdot \PP\Delta u. 
\end{align*}
Using  the cancellation
\begin{equation*}
 \int_\Om u\cdot\nabla\PP\Delta u\cdot \PP\Delta u=0 
\end{equation*}
we infer that
\begin{equation}\label{h2diffeq}
\frac12\partial_t(\nl2{\nabla u}^2+\al\nl2{\PP\Delta u}^2)+\nu\nl2{\PP\Delta u}^2  
=\int_\Om u\cdot\nabla u\cdot \PP\Delta u 
-\al\int_\Om \sum_j(\PP\Delta u)_j\nabla u_j \cdot \PP\Delta u. 
\end{equation}

We bound
\begin{equation}\label{h2diffeq2}
- \al\int_\Om \sum_j(\PP\Delta u)_j\nabla u_j \cdot \PP\Delta u
\leq \al\nl2{\PP\Delta u}^2\lip u
\leq \frac\nu2 \nl2{\PP\Delta u}^2
\end{equation}
and
\begin{align*}
\int_\Om u\cdot\nabla u\cdot \PP\Delta u 
&\leq \nl3u\nl6{\nabla u}\nl2{\PP\Delta u}\\
&\leq C\nl3u\nh1{\nabla u} \nl2{\PP\Delta u} \\
&\leq C\nl3u\nl2{\PP\Delta u}^2\\
&\leq C\ep_1\nu  \nl2{\PP\Delta u}^2\\
&\leq \frac\nu4\nl2{\PP\Delta u}^2
\end{align*}
provided that $\ep_1\leq1/4C$ which is the only condition we will impose on $\ep_1$. We used above the Sobolev embedding $H^1\subset L^6$ and the classical regularity result for the stationary Stokes operator which claims that $\nl2{\PP\Delta u}\simeq\nh2u$.

We infer from the previous relations that
\begin{equation*}
\partial_t(\nl2{\nabla u}^2+\al\nl2{\PP\Delta u}^2)+\frac\nu2\nl2{\PP\Delta u}^2  \leq0
\end{equation*}
so, after an integration in time,
\begin{equation}\label{h2est}
\nl2{\nabla u(t)}^2+\al\nl2{\PP\Delta u(t)}^2+\frac\nu2\int_0^t\nl2{\PP\Delta u(s)}^2\,ds  \leq \nl2{\nabla u_0}^2+\al\nl2{\PP\Delta u_0}^2.  
\end{equation}

\subsection*{$H^3$ estimates.}
Let us introduce the notation
\begin{equation*}
\om=\curl u,\qquad \oma=\om-\al\Delta \om.  
\end{equation*}

To perform the $H^3$ estimates, we apply the curl operator to \eqref{maineq}. We obtain the following PDE:
\begin{equation*}
\partial_t \oma-\nu\Delta\om+u\cdot\nabla\oma-\oma\cdot\nabla u=0.  
\end{equation*}

We multiply by $\oma$ and integrate in space to obtain
\begin{align*}
  \frac12\partial_t\nl2\oma^2-\nu\int_\Om\Delta\om\cdot\oma
&=\int_\Om\oma\cdot\nabla u\cdot\oma\\
&\leq\nl2\oma^2\lip u\\
&\leq\frac\nu{2\al}\nl2\oma^2.\notag
\end{align*}
Next, since $\oma=\om-\al\Delta\om$ we can write
\begin{equation*}
-\nu\int_\Om\Delta\om\cdot\oma
= -\frac\nu\al\int_\Om(\om-\oma)\cdot\oma
=\frac{3\nu}{4\al}\nl2\oma^2+\frac\nu\al\bigl\|\frac\oma2-\om\bigr\|^2_{L^2}-\frac\nu\al\nl2\om^2.
\end{equation*}

We infer that 
\begin{equation*}
\frac12\partial_t\nl2\oma^2 +\frac{3\nu}{4\al}\nl2\oma^2+\frac\nu\al\bigl\|\frac\oma2-\om\bigr\|^2_{L^2}-\frac\nu\al\nl2\om^2 \leq\frac\nu{2\al}\nl2\oma^2
\end{equation*}
so
\begin{equation*}
\partial_t\nl2\oma^2+\frac\nu{2\al}\nl2\oma^2\leq \frac{2\nu}\al\nl2\om^2= \frac{2\nu}\al\nl2{\nabla u}^2.  
\end{equation*}

The Gronwall inequality implies that
\begin{align}
\nl2{\oma(t)}^2
&\leq \nl2{\oma(0)}^2e^{-t\frac\nu{2\al}}
+\frac{2\nu}\al\int_0^te^{(s-t)\frac\nu{2\al}} \nl2{\nabla u(s)}^2\,ds\notag\\
&\leq \nl2{\oma(0)}^2e^{-t\frac\nu{2\al}}
+\frac{2\nu}\al\sup_{s\in[0,t]}\nl2{\nabla u(s)}^2\int_0^te^{(s-t)\frac\nu{2\al}} \,ds\notag\\ 
&\leq \nl2{\oma(0)}^2
+4\sup_{s\in[0,t]}\nl2{\nabla u(s)}^2\label{h3diffeq2}\\
&\leq \nl2{\oma(0)}^2
+4\nl2{\nabla u_0}^2+4\al\nl2{\PP\Delta u_0}^2\notag
\end{align}
where we used the bound for $\nl2{\nabla u}$ given in \eqref{h2est}. Adding this relation to \eqref{h2est} implies that 
\begin{equation}\label{finalbound}
F(t)\leq CF(0)  
\end{equation}
where
\begin{equation*}
  F(t)=\nl2{\nabla u(t)}+\sqrt\al\nl2{\PP\Delta u(t)}+\nl2{\oma(t)}.
\end{equation*}

The following lemma holds true:
\begin{lemma}\label{lemma1}
We have that
\begin{equation*}
F(t)\simeq \nh1{u(t)}+\al\nh3{u(t)}  
\end{equation*}
with constants depending only on $\Om$.
\end{lemma}
\begin{proof}
Because $\PP$ is bounded in $L^2$ and $\oma=\curl u-\al\Delta\curl u$ we immediately get the bound
\begin{equation*}
F\leq C(\nh1{u}+\sqrt\al\nh2{u}+\al\nh3{u}).   
\end{equation*}
From the interpolation inequality $\nh2u\leq \nh1u^{\frac12}\nh3u^{\frac12}$ we deduce that $\sqrt\al\nh2u\leq \nh1u+\al\nh3u$ so
\begin{equation*}
F\leq C(\nh1{u}+\al\nh3{u}).   
\end{equation*}

To prove the reverse bound we observe first by the Poincaré inequality that $\nl2{\nabla u}\simeq\nh1u$.
Therefore, it suffices to show that
\begin{equation*}
F\geq C\al\nh3u.  
\end{equation*}

Observing that $\nl2{\nabla u}=\nl2\om$ we can write
\begin{equation*}
\nl2\oma=\nl2{\om-\al\Delta\om}\geq\al\nl2{\Delta \om}-\nl2\om =\al\nl2{\curl\Delta u}-\nl2{\nabla u}. 
\end{equation*}
Next, we recall that $\Delta u-\PP\Delta u$ is a gradient so $\curl\Delta u=\curl\PP\Delta u$. We infer from the previous relations that
\begin{equation*}
F\geq \sqrt\al\nl2{\PP\Delta u}+\al\nl2{\curl\PP\Delta u} 
\geq \al(\nl2{\PP\Delta u}+\nl2{\curl\PP\Delta u} )
\end{equation*}
where we used that $\al\leq1$. Since $\PP\Delta u$ is divergence free and tangent to the boundary, one can apply \cite[Proposition 1.4]{foias_remarques_1978} to deduce that
\begin{equation*}
 \nl2{\PP\Delta u}+\nl2{\curl\PP\Delta u}\geq C\nh1{\PP\Delta u}. 
\end{equation*}
Finally, the classical regularity result for the stationary Stokes operator says in particular that
\begin{equation*}
 \nh1{\PP\Delta u}\geq C\nh3u. 
\end{equation*}
We infer that $F\geq C\al\nh3u$ and this completes the proof of the lemma.
\end{proof}

We go back to the proof of Theorem \ref{globalthm}. From the Gagliardo-Nirenberg estimate
\begin{equation*}
  \lip u\leq C\nh1u^{\frac14}\nh3u^{\frac34},
\end{equation*}
from relation \eqref{finalbound} and from Lemma \ref{lemma1} we infer that
\begin{equation*}
  \al\lip {u(t)}\leq C\al^{\frac14}F(t)\leq C\al^{\frac14}F(0)\leq C(\al^{\frac14}\nh1{u_0}+\al^{\frac54}\nh3{u_0}).
\end{equation*}

From the Gagliardo-Nirenberg and Poincaré inequalities and from relations \eqref{h1est} and \eqref{h2est} we can bound 
\begin{equation*}
\nl3u\leq C\nl2u^{\frac12}\nh1u^{\frac12}\leq C\nl2u^{\frac12}\nl2{\nabla u}^{\frac12}\leq C (\nl2{u_0}^2+\al\nl2{\nabla u_0}^2)^{\frac14}( \nl2{\nabla u_0}^2+\al\nl2{\PP\Delta u_0}^2)^{\frac14}.
\end{equation*}

We conclude that for \eqref{condimpl} to be verified it suffices to assume the following conditions on the initial data:
\begin{gather}
C(\al^{\frac14}\nh1{u_0}+\al^{\frac54}\nh3{u_0})\leq \frac\nu{4}\label{small11}\\
\intertext{and}
 C (\nl2{u_0}^2+\al\nl2{\nabla u_0}^2)^{\frac14}( \nl2{\nabla u_0}^2+\al\nl2{\PP\Delta u_0}^2)^{\frac14}\leq \frac{\ep_1\nu}2.  \label{small22}
\end{gather}

We prove now that if our smallness assumptions \eqref{small1} and \eqref{small2} hold true for a sufficiently small constant $\ep$, then \eqref{small11} and \eqref{small22} are verified. Observe first that \eqref{small2} implies \eqref{small11}. Indeed, this follows immediately from the interpolation inequality $\nh1{u_0}\leq\nl2{u_0}^{\frac12}\nh2{u_0}^{\frac12}$. To prove \eqref{small22}, we notice that it is equivalent to
\begin{equation*}
\nl2{u_0}\nh1{u_0}+\al^{\frac12}\nh1{u_0}^2+\al^{\frac12}\nl2{u_0}\nh2{u_0}+\al\nh1{u_0}\nh2{u_0}\leq\frac{\ep_1^2\nu^2}{C}  
\end{equation*}
for some constant $C$. The necessary bounds for the first, the third and the fourth term on the left-hand side are included in the hypothesis \eqref{small1} and \eqref{small2} for $\ep$ small enough. The bound for the second term  shows up in \eqref{small11}, so it is already proved.

We conclude that if \eqref{small1} and \eqref{small2} hold true for a sufficiently small constant $\ep$, then \eqref{cond} never fails to be true and the solution $u$ exists globally. This completes the proof of Theorem \ref{globalthm}.

\section{Uniform local existence of solutions}

We show in this section Theorem \ref{local}.

To prove the uniform local existence of solutions, we will adapt the estimates from the previous section in the following manner.

We search for a constant $M$ and a time $T$ such that if we assume that
\begin{equation}\label{concl0}
  \nl2{\nabla u}\leq M\text{ and }\nl\infty{\nabla u}\leq\frac\nu{2\al}
\end{equation}
on some sub-interval $[0,T']\subset[0,T]$ then
\begin{equation}\label{concl}
  \nl2{\nabla u}\leq \frac M2\text{ and }\nl\infty{\nabla u}\leq\frac\nu{4\al}
\end{equation}
on the same sub-interval $[0,T']$. By time continuity it follows that relation \eqref{concl} holds true on $[0,T]$ if it holds true at the initial time. As in the previous section one easily checks that the second condition from \eqref{concl0} is verified at time $t=0$ if $\ep$ is sufficiently small. Indeed, the condition on the Lipschitz norm of $u_0$ is the same  as in \eqref{cond} and in Section 3 we proved it using bounds for $\nh1{u_0}$ and $\nh3{u_0}$ which are the same as those from \eqref{hyplocal}. For the first bound in \eqref{concl0} to hold true at time $t=0$ we impose the following condition on the constant $M$ (recall that $M$ will be chosed later):
\begin{equation*}
\nl2{\nabla u_0}< M.  
\end{equation*}

So let us assume that \eqref{concl0} holds true and let us prove relation \eqref{concl}. We adapt the estimates from the previous section.
The $H^1$ estimates are not needed anymore. The $H^2$ estimates must be modified. Relations \eqref{h2diffeq} and \eqref{h2diffeq2} remain valid and imply that
\begin{equation}\label{loch2}
 \partial_t(\nl2{\nabla u}^2+\al\nl2{\PP\Delta u}^2)+\nu\nl2{\PP\Delta u}^2  
\leq 2\int_\Om u\cdot\nabla u\cdot \PP\Delta u. 
\end{equation}

The estimate of the right-hand side must be modified in the following manner
\begin{align*}
2\int_\Om u\cdot\nabla u\cdot \PP\Delta u 
&\leq 2\nl\infty u\nl2{\nabla u}\nl2{\PP\Delta u}\\
&\leq C\nh1u^{\frac32}\nh2u^{\frac12} \nl2{\PP\Delta u} \\
&\leq C\nl2{\nabla u}^{\frac32} \nl2{\PP\Delta u}^{\frac32}\\
&\leq  \frac\nu2 \nl2{\PP\Delta u}^2+\frac{C_1}{\nu^3}\nl2{\nabla u}^6.
\end{align*}
for some constant $C_1$. We used the Poincaré and Young inequalities, the Gagliardo-Nirenberg estimate $\nl\infty u\leq C\nh1u^{\frac12}\nh2u^{\frac12}$ and the relation $\nh2u\leq C\nl2{\PP\Delta u}$. Using this estimate in \eqref{loch2} yields
\begin{equation*}
 \partial_t(\nl2{\nabla u}^2+\al\nl2{\PP\Delta u}^2)+\frac\nu2\nl2{\PP\Delta u}^2  
\leq \frac{C_1}{\nu^3}\nl2{\nabla u}^6.
\end{equation*}
Recalling that we assumed $\nl2{\nabla u}\leq M$ and integrating in time yields
\begin{equation*}
\nl2{\nabla u(t)}^2+\al\nl2{\PP\Delta u(t)}^2\leq \nl2{\nabla u_0}^2+\al\nl2{\PP\Delta u_0}^2+\frac{C_1tM^6}{\nu^3}  
\end{equation*}
If we make the  assumption that
\begin{equation}\label{boundm2}
\nl2{\nabla u_0}^2+\al\nl2{\PP\Delta u_0}^2+\frac{C_1TM^6}{\nu^3}\leq \frac{M^2}4  
\end{equation}
then we get that
\begin{equation}\label{17a}
\nl2{\nabla u(t)}^2+\al\nl2{\PP\Delta u(t)}^2\leq \frac{M^2}4.    
\end{equation}
In particular, we have that
\begin{equation*}
  \nl2{\nabla u(t)}\leq \frac M2
\end{equation*}
which implies the first half of \eqref{concl}.

The $H^3$ estimates remain valid up to the relation \eqref{h3diffeq2} that we recall now
\begin{equation*}
  \nl2{\oma(t)}^2\leq\nl2{\oma(0)}^2+4\sup_{s\in[0,t]}\nl2{\nabla u(s)}^2.
\end{equation*}
We can further bound $\nl2{\nabla u(s)}\leq M$ to obtain that
\begin{equation*}
 \nl2{\oma(t)}^2\leq\nl2{\oma(0)}^2+4M^2. 
\end{equation*}
Adding this estimate to \eqref{17a} yields
\begin{equation*}
  F(t)\leq CF(0)+CM.
\end{equation*}
We observed in the previous section that $\lip u\leq C\al^{-\frac34}F$ so we deduce that
\begin{equation*}
\lip{u(t)}\leq  C\al^{-\frac34}(F(0)+M)
\leq C_2 (\al^{-\frac34}\nh1{u_0}+\al^{\frac14}\nh3{u_0}+\al^{-\frac34}M)
\end{equation*}
for some constant $C_2$, where we used Lemma \ref{lemma1}.

We conclude that if we assume \eqref{boundm2} together with 
\begin{equation}\label{boundm3}
C_2 (\al^{-\frac34}\nh1{u_0}+\al^{\frac14}\nh3{u_0}+\al^{-\frac34}M)\leq\frac\nu{4\al}  
\end{equation}
then \eqref{concl} holds true.

To satisfy the condition \eqref{boundm2} it suffices to assume that
\begin{equation}\label{f1}
\nl2{\nabla u_0}\leq \frac M{2\sqrt3}\qquad\text{and}\qquad \sqrt\al\nl2{\PP\Delta u_0}\leq \frac M{2\sqrt3}  
\end{equation}
and to choose $T$ such that
\begin{equation*}
  \frac{C_1TM^6}{\nu^3}=\frac{M^2}{12}
\end{equation*}
that is
\begin{equation}\label{f3}
  T=\frac{\nu^3}{12C_1M^4}.
\end{equation}

On the other hand, to ensure that \eqref{boundm3} holds true it suffices to assume that
\begin{equation}\label{f2}
  \nh1{u_0}\leq \frac\nu{12C_2\al^{\frac14}}, \qquad  \nh3{u_0}\leq \frac\nu{12C_2\al^{\frac54}}\qquad\text{and}\qquad M\leq \frac\nu{12C_2\al^{\frac14}}.
\end{equation}

Clearly, to be able to find a constant $M$ such that \eqref{f1} and \eqref{f2} are satisfied we need to assume the following conditions on the initial data:
\begin{equation}\label{condi}
\nh1{u_0}\leq \ep\nu\al^{-\frac14}, \qquad   \nh2{u_0}\leq \ep\nu\al^{-\frac34}\qquad\text{and}\qquad  \nh3{u_0}\leq \ep\nu\al^{-\frac54} 
\end{equation}
for a sufficiently small $\ep$.
In fact, the condition on the $H^2$ norm is not necessary since the interpolation inequality $\|\cdot\|_{H^2}\leq\|\cdot\|_{H^1}^{\frac12}\|\cdot\|_{H^3}^{\frac12}$ shows that it follows from the other two conditions.

Once  \eqref{condi} is satisfied, we can choose
\begin{equation*}
M=C_3(\nh1{u_0}+\sqrt\al\nh2{u_0})  
\end{equation*}
for some constant $C_3$. According to \eqref{f3}, this gives a time of existence of the solution of the form
\begin{equation*}
T=\frac{\nu^3}{12C_1C_3^4 (\nh1{u_0}+\sqrt\al\nh2{u_0})^4 }. 
\end{equation*}
This completes the proof of Theorem \ref{local}.

\section*{Acknowledgments}

This work was supported by the LABEX MILYON (ANR-10-LABX-0070) of Université de Lyon, within the program "Investissements d'Avenir" (ANR-11-IDEX-0007) operated by the French National Research Agency (ANR).

\end{document}